 \documentclass[final,1p,times]{elsarticle}
\usepackage{amsmath,amssymb,url}
\usepackage{enumitem} 
\numberwithin{equation}{section}

\usepackage{amssymb}
\usepackage{amsthm}
\usepackage{hyperref}
\usepackage{amsmath,amssymb,amsopn,amsfonts,mathrsfs,amsbsy,amscd}
\usepackage{longtable}
\journal{ .}

\newtheorem{definition}{Definition}[section]
\newtheorem{theorem}{Theorem}[section]
\newtheorem{proposition}{Proposition}[section]
\newtheorem{lemma}{Lemma}[section]
\newtheorem{corollary}{Corollary}[section]

\numberwithin{equation}{section}

\begin{document}
\bibliographystyle{amsplain}
\begin{frontmatter}
\title{$\frac{1}{2}$-derivations and biderivations of conservative algebras of 2-dimensional algebras}

\author[label2]{Hassan Oubba}
\address[label2]{ Laboratoire Mathématiques Fondamentales et Appliquées\\ Faculty of Sciences, Ain Chock\\ Hassan II University of Casablanca\\
B.P. 5366 Maarif, Casablanca, Morocco\\
{HASSAN.OUBBA@univh2c.ma}}

\begin{abstract}
In the present paper, we prove that every local and 2-local $\frac{1}{2}$-derivation on a conservative algebra of a two-dimensional algebra is a $\frac{1}{2}$-derivation. Furthermore, we show that conservative algebras of two-dimensional algebras admit no nontrivial biderivations. We also study similar problems for the algebra $\mathcal{W}_2$ of all commutative algebras on the two-dimensional vector space and the algebra
$\mathcal{S}_2$ of all commutative algebras with trace zero multiplication on the two-dimensional space.\\
\textbf{Mathematics Subject Classification }15A03,
15A69, 17A30, 17B40, 17B63, 17D10.
\end{abstract}
\begin{keyword}
Derivation, $\frac{1}{2}$-derivation, local $\frac{1}{2}$-derivation, $2$-local $\frac{1}{2}$-derivation  biderivation , conservative algebra. 
\end{keyword}
\end{frontmatter}
\section*{Introduction} \label{Introduction}
We consider an arbitrary field \( \mathbb{F} \) of characteristic zero. Note that an algebra in this context may lack a unity element and may not necessarily be associative. A **multiplication** on a vector space \( \mathcal{W} \) is defined as a bilinear map \( \mathcal{W} \times \mathcal{W} \to \mathcal{W} \).

Given,  a vector space \( \mathcal{W} \), a linear mapping \( M \) on \( \mathcal{W} \),and a bilinear mapping \( N : \mathcal{W} \times \mathcal{W} \to \mathcal{W} \),
we can define a new multiplication \( [M, N] : \mathcal{W} \times \mathcal{W} \to \mathcal{W} \) as follows:
\[
[M,N](u, v) = M(N(u, v)) - N(M(u), v) - N(u,M(v))
\]
for all \( u, v \in \mathcal{W} \).For an algebra \( A \) with a multiplication operation \( P \) and an element \( x \in A \), we denote by \( L_P x \) the left multiplication of \( x \) by \( P \). When the multiplication \( P \) is fixed, we will simply write \( L x \) in place of \( L_P x \).

In 1990, Kantor \cite{Kan2} introduced a multiplication operation, denoted by \( \cdot \), for the set of all algebras (i.e., all possible multiplications) on an \( n \)-dimensional vector space \( V_n \). This multiplication is defined as follows:
\[
M \cdot N = [L_M e, N],
\]
where \( M \) and \( N \) are multiplications and \( e \) is the fixed vector from \( V_n \). In case \( n > 1 \), the resulting algebra \( \mathcal{W}(n) \) does not belong to the well-known classes of algebras (such as associative, Lie, Jordan, Leibniz algebras). The algebra \( \mathcal{W}(n) \) turns out to be a conservative algebra (see \cite{Kay, Iv1, Iv2, Iv3, Arz}).

In 1972, Kantor \cite{Kan1} introduced conservative algebras as a generalization of Jordan algebras. Specifically, an algebra \( A \) with multiplication defined as \( P(x, y) = xy \) and underlying vector space \( \mathcal{W} \) is called a conservative algebra if a new multiplication \( F \) can be defined on \( \mathcal{W} \) such that
\[
[L_b, [L_a, P]] = -[L_{F(a,b)}, P] \tag{1}
\]
for all \( a, b \in W \). In other words, the following identity holds for all \( a, b, x, y \) from \( \mathcal{W} \):
\begin{eqnarray*}
-F(a, b)(xy) + (F(a, b)x)y + x(F(a, b)y&=&b(a(xy) - (ax)y - x(ay)) - a((bx)y)\\
&&+ (a(bx))y + (bx)(ay) - a(x(by)) + (ax)(by) + x(a(by)). \quad (2)
\end{eqnarray*}
The algebra with the multiplication \( F \) is said to be associated with \( A \). For example, it is straightforward to observe that every 4-nilpotent algebra is a conservative algebra, with \( F(a, b) = 0 \).

The algebra \( \mathcal{W}(n) \) plays a similar role in the theory of conservative algebras as the Lie algebra of all \( n \times n \) matrices, \( \mathfrak{gl}_n \), plays in the theory of Lie algebras. Specifically, in \cite{Kan2}, Kantor considered the category \( \mathcal{S}_n \) whose objects are conservative algebras of non-Jacobi dimension \( n \). It was proven that the algebra \( \mathcal{W}(n) \) is the universal attracting object in this category; that is, for every algebra \( M \) in \( \mathcal{S}_n \), there exists a canonical homomorphism from \( M \) into the algebra \( \mathcal{W}(n) \).

In particular, all Jordan algebras of dimension \( n \) with unity are contained within \( \mathcal{W}(n) \). The same holds true for all noncommutative Jordan algebras of dimension \( n \) with unity. Some properties of the product in the algebra \( \mathcal{W}(n) \) were studied in \cite{Kay}.

The study of local derivations dates back to the works of R. V. Kadison \cite{Kad} and D. R. Larson and A. R. Sourour \cite{Lar}, published in 1990. A few years later, in 1997, P. Šemrl \cite{Sem} further developed this line of research by introducing the notions of 2-local derivations and 2-local automorphisms. Since then, one of the main objectives in the study of these mappings has been to characterize classes of algebras for which every local or 2-local derivation is necessarily a derivation in the usual sense. Another important direction is the construction and investigation of examples of algebras admitting local or 2-local derivations that are not derivations.\\
The study of local and 2-local derivations on conservative algebras of two-dimensional algebras was carried out by F. Arzikulov and U. Umrzaqov in \cite{Arz}. In their work, they proved that such algebras admit no nontrivial local or 2-local derivations.\\
The study of local and 2-local $\delta$-derivations of Lie algebras was initiated by A. Khudoyberdiyev and B. Yusupov in \cite{Khu}. In this work, the authors introduced the concepts of local and 2-local $\delta$-derivations and investigated local and 2-local $\frac{1}{2}$-derivations for finite-dimensional solvable Lie algebras whose nilradicals are filiform, Heisenberg, or abelian. They also characterized local $\frac{1}{2}$-derivations for several important classes of Lie algebras, including oscillator Lie algebras, conformal perfect Lie algebras, and Schrödinger algebras. A related question in the setting of Leibniz algebras was subsequently studied by U. Mamadaliyev, A. Sattarov, and B. Yusupov in \cite{Mama}. In the present paper, we continue the study of $\frac{1}{2}$-derivations, local and 2-local $\frac{1}{2}$-derivations of conservative algebras of 2-dimensional algebras.

Bre$\check{s}$ar et al. introduced the notion of biderivation of rings in \cite{br0}, and they showed that all biderivations of noncommutative prime rings are inner. As is well-known, any biderivation can be decomposed into a sum of a skew-symmetric biderivation and a symmetric biderivation. In recent years, many scholars gave their attentions to the study of biderivations
of many (Lie) algebras using case by case discussion, see \cite{Baj, bn, br0, br1, du, ob0, ob1, ob2, Oub}. There is no uniform method to determine all biderivations of some classes of (Lie)
algebras.

 This paper is organized as follows. In the first section we recall the necessary information concerning conservative algebra of 2-dimensional algebras and we give a characterization of $\frac{1}{2}$-derivation, local (resp. $2$-local) $\frac{1}{2}$-derivation and biderivation  of $\mathcal{W}(2)$. Section 2 is dedicated to characterization of $\frac{1}{2}$-derivation, local (resp. $2$-local) $\frac{1}{2}$-derivation and biderivation  of terminal algebra $\mathcal{W}_2$. In section 3 we give a characterization of $\frac{1}{2}$-derivation, local (resp. $2$-local) $\frac{1}{2}$-derivation and biderivation  of  $\mathcal{S}_2$.
\section{Conservative algebra of 2-dimensional algebra}
We consider the space \( \mathcal{W}(2) \) of all multiplications on the 2-dimensional space \( V_2 \), whit basis \( v_1, v_2 \). In \cite{Kayg1}, Kaygorodov et al. describe a basis and the multiplication table of the conservative algebra \( \mathcal{W}(2) \) as follows:

\[
\begin{array}{c|cccccccc}
\cdot & e_1 & e_2 & e_3 & e_4 & e_5 & e_6&e_7&e_8 \\
\hline
e_1&-e_1&0&0&e_4 &-2e_5& -e_6& -e_7&0\\
e_2&-e_2-e_3& -e_4 & -e_4&0&e_1-e_6-e_7& e_2-e_8& e_3-e_8&e_4\\
e_3&0&0&0&0&0&0&0&0\\
e_4&0&0&0&0&0&0&0&0\\
e_5&e_5&-e_1+e_6&-e_1+e_7&-e_2-e_3+e_8&0& -e_5&-e_5& -e_6-e_7\\
e_6&0&-e_2&-e_3&-2e_4& e_5&0&0&-e_8\\
e_7&0&0&0&0&0&0&0&0\\
e_8&0&0&0&0&0&0&0&0\\
\end{array}
\]
\subsection{$\frac{1}{2}$-derivations of $\mathcal{W}(2)$}
In this subsection, we characterize $\frac{1}{2}$-derivation, Local and $2$-local $\frac{1}{2}$-derivation of $\mathcal{W}(2)$.\\

Let us recall definition of $\frac{1}{2}$-derivation and centroid of an arbitrary algebra.
\begin{definition}
\label{def1}
Let $\mathcal{A}$ be an algebra. A  linear map $d: \mathcal{A} \rightarrow \mathcal{A}$ is called $\frac{1}{2}$-derivation if  
$$d(x.y)=\frac{1}{2}(d(x).y+x.d(y)),$$
for all $x,y \in \mathcal{A}$.
\end{definition}
We denoted the set of all $\frac{1}{2}$-derivations of an algebra $\mathcal{A}$  by $\Delta(\mathcal{A})$.
\begin{definition}
\label{def2}
Let $\mathcal{A}$ be an algebra. The set 
$$\Gamma(\mathcal{A})=\left\lbrace  \gamma \in End(\mathcal{A}) \hspace{0.2 cm} : \hspace{0.2 cm} \gamma(x.y)=\gamma(x).y=x.\gamma(y), \hspace{0.2 cm} \forall x,y \in \mathcal{A} \right\rbrace $$
is called the centroid of $\mathcal{A}$.
\end{definition}

 It is a simple observation to see that $\Gamma(\mathcal{A}) \subseteq \Delta(\mathcal{A})$ and $\Gamma(\mathcal{A})$ is closed under composition and thus has Lie algebra structure. Hence, we can consider $\Gamma(\mathcal{A})$ as Lie subalgebra of $End(\mathcal{A})$.\\
 We recall also the notion of local and $2$-local $\frac{1}{2}$-derivation.
\begin{definition}
A linear map $D : \mathcal{A} \rightarrow \mathcal{A}$ is called a local $\frac{1}{2}$-derivation if for all $x \in \mathcal{A}$ there exists a $\frac{1}{2}$-derivation $d_x$ such that $D(x)=d_x(x)$.
\end{definition}
\begin{definition}
A (not necessary linear) map $D : \mathcal{A} \rightarrow \mathcal{A}$ is called a $2$-local $\frac{1}{2}$-derivation if for all $x,y \in \mathfrak{A}$ there exists a $\frac{1}{2}$-derivation $d_{x,y}$ such that $D(x)=d_{x,y}(x)$ and $D(y)=d_{x,y}(y)$.
\end{definition}
Now, we give our first main result.
\begin{theorem}\label{der}
Let $d:\mathcal{W}(2) \rightarrow \mathcal{W}(2)$ be a linear map. Then $d$ is a $\frac{1}{2}$-derivation if and only if there exists a scalar $\lambda \in \mathbb{K}$ such that  
$$d(x)=\lambda x , \quad \forall x\in \mathcal{W}(2).$$
\end{theorem}
\begin{proof}
Given a $\frac{1}{2}$-derivation $d$ of $\mathcal{W}(2)$, we have $d(e_j)=\sum a_{ij}e_i$ for some $a_{ij}\in \mathbb{K}$. We consider some relations between images of $e_1,\dots,e_8$ with respect to $d$ obtain a system of linear equations on $\lbrace a_{ij}\rbrace$.

Since $-d(e_1)=d(e_1.e_1)=\frac{1}{2}(d(e_1).e_1+e_1.d(e_1))$, then
$$a_{21}=a_{31}=a_{41}=a_{51}=a_{61}=a_{71}=a_{81}=0.$$
Since $d(e_4)=d(e_1.e_4)=\frac{1}{2}(d(e_1).e_4+e_1.d(e_4))$, we have 
$$a_{14}=a_{24}=a_{34}=a_{54}=a_{64}=a_{74}=a_{84}=0 \quad \mbox{and} \quad a_{44}=a_{11}.$$
Since $-2d(e_5)=d(e_1.e_5)=\frac{1}{2}(d(e_1).e_5+e_1.d(e_5))$, we have 
$$a_{15}=a_{25}=a_{35}=a_{45}=a_{65}=a_{75}=a_{85}=0 \quad \mbox{and} \quad a_{55}=a_{11}.$$
Since $-d(e_6)=d(e_1.e_6)=\frac{1}{2}(d(e_1).e_6+e_1.d(e_6))$, we obtain
$$a_{16}=a_{26}=a_{36}=a_{46}=a_{76}=a_{86}=0 \quad \mbox{and} \quad a_{66}=a_{11}.$$

Since $0=d(e_1.e_2)=\frac{1}{2}(d(e_1).e_2+e_1.d(e_2))$, we obtain
$$a_{12}=a_{42}=a_{52}=a_{62}=a_{72}=0 .$$

Since $-d(e_2)=d(e_6.e_2)=\frac{1}{2}(d(e_6).e_2+e_6.d(e_2))$, we obtain
$$a_{56}=a_{32}=a_{82}=0 \quad \mbox{and}\quad a_{22}=a_{11}.$$

Since $0=d(e_1.e_3)=\frac{1}{2}(d(e_1).e_3+e_1.d(e_3))$, we obtain
$$a_{13}=a_{43}=a_{53}=a_{63}=a_{73}=0 .$$

Since $-d(e_3)=d(e_6.e_3)=\frac{1}{2}(d(e_6).e_3+e_6.d(e_3))$, we obtain
$$a_{56}=a_{23}=a_{83}=0 \quad \mbox{and} \quad a_{33}=a_{11}.$$

Since $-d(e_7)=d(e_1.e_7)=\frac{1}{2}(d(e_1).e_7+e_1.d(e_7))$, we obtain
$$a_{17}=a_{27}=a_{37}=a_{47}=a_{67}=a_{87}=0 \quad \mbox{and} \quad a_{77}=a_{11}.$$

Since $-d(e_8)=d(e_6.e_8)=\frac{1}{2}(d(e_6).e_8+e_6.d(e_8))$, we obtain
$$a_{18}=a_{28}=a_{38}=a_{58}=a_{68}=a_{78}=0 \quad \mbox{and} \quad a_{88}=a_{11}.$$

Since $0=d(e_7.e_8)=\frac{1}{2}(d(e_7).e_8+e_7.d(e_8))$, we obtain
$$a_{57}=0 .$$

Since $0=d(e_1.e_8)=\frac{1}{2}(d(e_1).e_8+e_1.d(e_8))$, we obtain
$$a_{48}=0 .$$
Thus, we can easily see that
$$a_{11}=a_{22}=a_{33}=a_{44}=a_{55}=a_{66}=a_{77}=a_{88}=\lambda$$
for some $\lambda \in \mathbb{K}$. Moreover, any other element $a_{ij}$ is equal to zero. Then, $d(e_i)=\lambda e_i$ for all $i=1,\dots,8$. i.e $d(x)=\lambda x$ for all $x \in \mathcal{W}(2)$.
\end{proof}
As application of the above theorem we have the following result. 
\begin{corollary}
Let $\mathcal{W}(2)$ be a conservative algebra of $2$-dimensional algebras. Then
$$\Gamma(\mathcal{W}(2))=\lbrace \lambda id \, :\, \lambda \in \mathbb{K} \rbrace.$$
\end{corollary}

Now, we are going to prove that any local $\frac{1}{2}$-derivation of $\mathcal{W}(2)$, $\mathcal{W}_2$ and $\mathcal{S}_2$ is a $\frac{1}{2}$-derivation.

\begin{theorem}\label{loc}
Every local $\frac{1}{2}$-derivation of $\mathcal{W}(2)$ is a $\frac{1}{2}$-derivation.
\end{theorem}
\begin{proof}
Let $D$ be a local $\frac{1}{2}$-derivation, for all $x \in \mathcal{W}(2)$ there exists a scalar $\lambda_x$ such that $D(x)=\lambda_x x$. \\
Taking, $x=e_1,\dots,8$ we have $D(e_i)=\lambda_{e_i} e_i$, then it sufficient to show that $\lambda_{e_i}=\lambda_{e_j}$ for $i,j=1,\dots,8$.\\
Since $D$ is linear, then $D(e_i+e_j)=D(e_i)+D(e_j)$. Which implies that $\lambda_{e_i}=\lambda_{e_i+e_j}=\lambda_{e_j}$. So $D(e_i)=\lambda e_i$ for all $i=1,\dots,8$. Therefore, by theorem\ref{der}, $D$ is a $\frac{1}{2}$-derivation of $\mathcal{W}(2)$. This ends the proof.
\end{proof}
Now, we give a characterization of 2-Local $\frac{1}{2}$-derivations on conservative algebras of $2$-dimensional algebras.

\begin{theorem}
Every 2-Local $\frac{1}{2}$-derivation of $\mathcal{W}(2)$, $\mathcal{W}_2$ and $\mathcal{S}_2$ is a $\frac{1}{2}$-derivation.
\end{theorem}
\begin{proof}
Let $\Delta$ be an arbitrary 2-local $\frac{1}{2}$-derivation of $\mathcal{W}(2)$. Then, by definition, for any element $a,\in \mathcal{W}(2)$, there exists a $\frac{1}{2}$-derivation $D_{a,e_1}$ of $\mathcal{W}(2)$ such that 
$$\Delta(a)=D_{a,e_1}(a), \, \mbox{and} \, \Delta(e_1)=D_{a,e_1}(e_1).$$
By theorem\ref{der}, there exists a scalar $\lambda^{a,e_1}$ such that $D_{a,e_1}=\lambda^{a,e_i}id$.\\
Let $b \in \mathcal{A}$ be an arbitrary element in $\mathcal{W}(2)$. Then, there exists a $\frac{1}{2}$-derivation of $\mathcal{W}(2)$ such that 
$$\Delta(b)=D_{b,e_1}(b), \, \mbox{and} \, \Delta(e_1)=D_{b,e_1}(e_1).$$ 
By theorem\ref{der}, there exists a scalar $\lambda^{b,e_1}$ such that $D_{b,e_1}=\lambda^{a,e_i}id$.\\ 
Since, $\Delta(e_1)=D_{a,e_1}(e_1)=D_{b,e_1}(e_1)$, we have
$$\lambda^{a,e_1}=\lambda^{b,e_1}$$
Thus 
$$D_{a,e_1}=D_{b,e_1}$$
Therefore, for any $a \in \mathcal{W}(2)$
$$\Delta(a)=D_{b,e_1}(a)=\lambda^{b,e_1}a,$$
which means that $D_{b,e_1}$ is independent of $a$. Hence, $\Delta$ is a $\frac{1}{2}$-derivation. This completes the proof.
\end{proof}
\subsection{ Biderivations of $\mathcal{W}(2)$}
In this section, we give a characterization of biderivations of conservative algebras of 2-dimensional algebras. We start by recalling the definition of biderivation of an arbitrary algebra.
\begin{definition}
\label{def3}
 Let $\mathcal{A}$ be an arbitrary algebra. A bilinear map $\delta : \mathcal{A} \times \mathcal{A} \rightarrow \mathcal{A}$ is called a biderivation of $\mathcal{A}$ if
 $$\delta(xy,z)=x\delta(y,z)+\delta(x,z)y,$$
 $$\delta(x,yz)=y\delta(x,z)+\delta(x,y)z,$$
 for all $x,y,z \in \mathcal{A}$.
 \end{definition}
Denote by $BDer(\mathcal{A})$ the set of all biderivations on $\mathcal{A}$. Clearly, $BDer(\mathcal{A})$ is a vector space.\\
A $\delta \in BDer(\mathcal{A})$ is called symmetric if $\delta(x,y)=\delta(y,x)$ for all $x,y \in \mathcal{A}$, and is called skew-symmetric if $\delta(x,y)=-\delta(y,x)$ for all $x,y \in \mathcal{A}$. Denote by $BDer_+(\mathcal{A})$ and $BDer_-(\mathcal{A})$ the subspace of all symmetric biderivations and all skew-symmetric biderivations on $\mathcal{A}$ respectively.\\
For any $\delta \in BDer(\mathcal{A})$, we define two bilinear maps by 
$$\delta^+(x,y)=\delta(x,y)+\delta(y,x), \hspace{0.3 cm}\delta^-(x,y)=\delta(x,y)-\delta(y,x).$$
It is easy to see that $\delta^+ \in BDer_+(\mathcal{A})$ and $\delta^- \in BDer_-(\mathcal{A})$. Since $\delta =\frac{1}{2}(\delta^+ + \delta^-)$ it follows that
$$BDer(\mathcal{A})=BDer_+(\mathcal{A}) \oplus BDer_-(\mathcal{A}).$$
To characterize $BDer(\mathcal{A})$, we only need to characterize $BDer_+(\mathcal{A})$ and $BDer_-(\mathcal{A})$.\\
Our main tool for study of biderivations of the algebras $\mathcal{W}(2)$  is the following lemma [\cite{Kayg1}, Theorem 6], where the matrix of a derivation is calculated in the new basis $\lbrace e_1,e_2,...,e_8 \rbrace$.  
\begin{lemma} \cite{Kayg1}
\label{lem1}
Let $d$ be a derivation of $\mathcal{W}(2)$. Then, the matrix $D_d$ of $d$ is of the form:
\begin{center}
$D_{d}= \begin{pmatrix}
		0 & a &0 &0&0&0&0&0\\
		0 & -b &0 &0&0&0&0&0\\
		2a & 0 &b &0&0&0&0&0\\
		0 & 0 &3a &2b&0&0&0&0\\
		0 & 0 &0 &0&0&0&0&0\\
		0 & 0 &0 &0&-a&b&0&0\\
		0 & 0 &0 &0&0&0&b&a\\
		0 & 0 &0 &0&0&0&0&0
		\end{pmatrix}. $
\end{center}
where $a,b \in \mathbb{F}$.
\end{lemma}
Let $\delta$ be a  biderivation of a conservative algebras of 2-dimensional Algebras $\mathcal{W}(2)$ and $x,y \in \mathcal{W}(2)$, such that $x=\sum_{i=1}^{8}x_{i}e_{i}$ and $y=\sum_{i=1}^{8} y_{i}e_{i}$. Then, by the bilinearity of $\delta^-$, we obtain, 
$$\delta(x,y)=\sum_{i=1}^{8} \sum_{j=1}^{8} x_{i} y_{j} \delta(e_i,e_j)=
\sum_{i=1}^{8} \sum_{j=1}^{8} x_{i} y_{j} \delta_{e_i}(e_j).$$ 
By Theorem \ref{lem1}, the matrix $D_{e_i}$ of $\delta_{e_i}$, for $i=1,2,...,8$ in the basis $\lbrace e_1,e_2,...,e_8 \rbrace$ is of the form
\begin{equation}
\label{eq0}
D_{e_i}= \begin{pmatrix}
		0 & a_i &a_i &0&0&0&0&0\\
		0 & b_i &0 &a_i&0&0&0&0\\
		0 & 0 &b_i &a_i&0&0&0&0\\
		0 & 0 &0 &2b_i&0&0&0&0\\
		-a_i & 0 &0 &0&-b_i&a_i&a_i&0\\
		0 & -a_i &0 &0&0&0&0&a_i\\
		0 & 0 &-a_i &0&0&0&0&a_i\\
		0 & 0 &0 &-a_i&0&0&0&b_i
		\end{pmatrix}. 
\end{equation}
We state our first main result
\begin{proposition}
\label{pro1}
The conservative algebra of 2-dimensional algebras $\mathcal{W}(2)$ has no nontrivial skew-symmetric biderivation.
\end{proposition}
\begin{proof}
Let $\delta^-$ be a skew-symmetric biderivation of $\mathcal{W}(2)$. Then, for $i=1,...,8$ the matrix of $\delta^-_{e_i}$ is of the form (\ref{eq0}).\\ 
Since $ \delta^-$ is skew-symmetric, then, $\delta^-(e_i,e_i)=0$ for $i=1,2,...,8$. Therefore,
$$a_1=a_2=b_2=a_3=b_3=a_4=b_4=b_5=a_6=a_7=a_8=b_8=0.$$
For $i=1,5,6,7$, the equality $\delta^-(e_i,e_2)=-\delta^-(e_2,e_i)$ implies that
$$b_1=a_5=b_6=b_7=0.$$
That is, for $i=1, \dots,8$. 
\begin{center}
$D_{e_i}= \begin{pmatrix}
		0 & 0 &0 &0&0&0&0&0\\
		0 & 0 &0 &0&0&0&0&0\\
		0 & 0 &0 &0&0&0&0&0\\
		0 & 0 &0 &0&0&0&0&0\\
		0 & 0 &0 &0&0&0&0&0\\
		0 & 0 &0 &0&0&0&0&0\\
		0 & 0 &0 &0&0&0&0&0\\
		0 & 0 &0 &0&0&0&0&0
		\end{pmatrix}. $
\end{center}
Therefore, $\delta^-(x,y)=0$ for all $x,y \in \mathcal{W}(2)$.
\end{proof}
\begin{proposition}
\label{pro2}
The conservative algebra of 2-dimensional algebras $\mathcal{W}(2)$ has no nontrivial symmetric biderivation.
\end{proposition}
\begin{proof}
Let $\delta^+$ be a skew-symmetric biderivation of $\mathcal{W}(2)$. Then, for $i=1,...,8$ the matrix of $\delta^+_{e_i}$ is of the form (\ref{eq0}).\\
For $i=1,5,6,7$, the equality $\delta^+(e_i,e_2)=\delta^+(e_2,e_i)$ implies that
$$a_1=b_1=a_2=b_2=a_5=b_5=a_6=b_6=a_7=b_7=0.$$
From the equality $\delta^+(e_3,e_4)=\delta^+(e_4,e_3)$, we get
$$a_3=b_3=a_4=b_4.$$
Therefore, for any $i=1,\dots, 8$. We have 
\begin{center}
$D_{e_i}= \begin{pmatrix}
		0 & 0 &0 &0&0&0&0&0\\
		0 & 0 &0 &0&0&0&0&0\\
		0 & 0 &0 &0&0&0&0&0\\
		0 & 0 &0 &0&0&0&0&0\\
		0 & 0 &0 &0&0&0&0&0\\
		0 & 0 &0 &0&0&0&0&0\\
		0 & 0 &0 &0&0&0&0&0\\
		0 & 0 &0 &0&0&0&0&0
		\end{pmatrix}. $
\end{center}
Therefore, $\delta^+(x,y)=0$ for all $x,y \in \mathcal{W}(2)$.
\end{proof}
Our second main result is now stated and proved in the following theorem.
\begin{theorem}
\label{thm2}
Let $\mathcal{W}(2)$ be a conservative algebra of 2-dimensional algebras. Then, $\delta$ is a biderivation of $\mathcal{W}(2)$ if and only if $\delta$ is a zero mapping. i.e 
$BDer(\mathcal{W}(2))=\lbrace 0 \rbrace$.
\end{theorem}
\section{Terminal algebra $\mathcal{W}_2$ of all $2$-dimensional commutative algebras}
Consider the subalgebra \( \mathcal{W}_2 \) of \( \mathcal{W}(2) \) that was introduced by Kantor in \cite{Kan3, Kan2} as the algebra of all bilinear commutative operations on the 2-dimensional space \( V_2 \) with a basis \( v_1, v_2 \). The algebra \( \mathcal{W}_2 \) is known to be terminal \cite{Kan3} and therefore it is conservative. Kaygorodov et al. prove in \cite{Kayg1} that the dimension of the algebra \( \mathcal{W}_2 \) is 6, and \( \lbrace \xi_1, \dots, \xi_6 \rbrace \) is a basis of this algebra. The multiplication table of the terminal algebra \( \mathcal{W}_2 \) is given by:

\[
\begin{array}{c|cccccc}
\cdot & \xi_1 & \xi_2 & \xi_3 & \xi_4 & \xi_5 & \xi_6 \\
\hline
\xi_1 & -\xi_1 & 0 & \xi_3 & -2\xi_4 & -\xi_5 & 0 \\
\xi_2 & -\xi_2 & -2\xi_3 & 0 & \xi_1 - \xi_5 & \xi_2 - 2\xi_6 & \xi_3 \\
\xi_3 & 0 & 0 & 0  & 0 & 0 & 0 \\
\xi_4 & \xi_4 & \xi_5-2\xi_1 & -\xi_2+\xi_6 & 0 & -2\xi_4 & -\xi_5 \\
\xi_5 & 0 & -\xi_2 & -2\xi_3 & \xi_4 & 0 & -\xi_6 \\
\xi_6 & 0 & 0 & 0 & 0 & 0 & 0 \\
\end{array}
\]
\subsection{$\frac{1}{2}$-derivations of $\mathcal{W}_2$}
In this subsection, we characterize $\frac{1}{2}$-derivation and Local (resp. $2$-local) $\frac{1}{2}$-derivation of $\mathcal{W}_2$.\\

The following theorem characterize $\frac{1}{2}$-derivation of $\mathcal{W}_2$.
\begin{theorem}\label{der2}
Let $d:\mathcal{W}_2 \rightarrow \mathcal{W}_2$ be a linear map. Then $d$ is a $\frac{1}{2}$-derivation if and only if there exists a scalar $\lambda \in \mathbb{K}$ such that  
$$d(x)=\lambda x , \quad \forall x\in \mathcal{W}_2.$$
\end{theorem}
\begin{proof}
Let $d$ be a $\frac{1}{2}$-derivation of $\mathcal{W}_2$ and $d(\xi_j)=\sum a_{ij}e_i$ for some $a_{ij}$ from $\mathbb{K}$.\\
Then the equality  $-d(\xi_1)=d(\xi_1.\xi_1)=\frac{1}{2}(d(\xi_1).\xi_1+\xi_1.d(\xi_1))$ implies that 
$$a_{21}=a_{31}=a_{41}=a_{51}=a_{61}=0.$$
Since $-d(\xi_2)=d(\xi_2.\xi_1)=\frac{1}{2}(d(\xi_2).\xi_1+\xi_2.d(\xi_1))$, we can see that
$$a_{12}=a_{32}=a_{42}=a_{52}=a_{62}\quad \mbox{and} \quad a_{22}=a_{11}.$$
Since $d(\xi_3)=d(\xi_1.\xi_3)=\frac{1}{2}(d(\xi_1).\xi_3+\xi_1.d(\xi_3))$, we obtain
$$a_{13}=a_{23}=a_{43}=a_{53}=a_{63}\quad \mbox{and} \quad a_{33}=a_{11}.$$
Since $-d(\xi_4)=d(\xi_1.\xi_4)=\frac{1}{2}(d(\xi_1).\xi_4+\xi_1.d(\xi_4))$, we get
$$a_{14}=a_{24}=a_{34}=a_{54}=a_{64} \quad \mbox{and} \quad a_{44}=a_{11}.$$
The equality $d(\xi_1-\xi_5)=d(\xi_2.\xi_4)=\frac{1}{2}(d(\xi_2).\xi_4+\xi_2.d(\xi_4))$ implies that 
$$a_{15}=a_{25}=a_{35}=a_{45}=a_{65} \quad \mbox{and} \quad a_{55}=a_{11}.$$
From the equality $d(\xi_6-\xi_2)=d(\xi_4.\xi_3)=\frac{1}{2}(d(\xi_4).\xi_3+\xi_4.d(\xi_3))$ we deduce
$$a_{16}=a_{26}=a_{36}=a_{46}=a_{56} \quad \mbox{and} \quad a_{66}=a_{11}.$$
Therefore, it is easy to see that 
$$a_{11}=a_{22}=a_{33}=a_{44}=a_{55}=a_{66}=\lambda,$$
and any other element $a_{ij}$ is equal to zero. In other words, $d(\xi_i)=\lambda \xi_i$ ($i=1,\dots,6$). The theorem is proven.
\end{proof}
\begin{corollary}
Let $\mathcal{W}_2$ be a terminal algebra, then 
$$\Gamma(\mathcal{W}_2)=\lbrace \lambda id \, :\, \lambda \in \mathbb{K} \rbrace.$$
\end{corollary}
We stat now the characterization of local and $2$-local  $\frac{1}{2}$-derivation of $\mathcal{W}_2$.
\begin{theorem}\label{loc2}
Every local $\frac{1}{2}$-derivation of $\mathcal{W}_2$ is a $\frac{1}{2}$-derivation.
\end{theorem}
\begin{proof}

Let $D$ be a local $\frac{1}{2}$-derivation, for all $x \in \mathcal{W}_2$ there exists a scalar $\lambda_x$ such that $D(x)=\lambda_x x$. \\
Taking, $x=\xi_1,\dots,\xi_6$ we have $D(\xi_i)=\lambda_{\xi_i} \xi_i$, then it sufficient to show that $\lambda_{\xi_i}=\lambda_{\xi_j}$ for $i,j=1,\dots,6$.\\
Since $D$ is linear, then $D(\xi_i+\xi_j)=D(\xi_i)+D(\xi_j)$. Which implies that $\lambda_{\xi_i}=\lambda_{\xi_i+\xi_j}=\lambda_{\xi_j}$. So $D(\xi_i)=\lambda \xi_i$ for all $i=1,\dots,6$. Therefore, by theorem\ref{der2}, $D$ is a $\frac{1}{2}$-derivation of $\mathcal{W}_2$. This ends the proof.
\end{proof}
\begin{theorem}\label{loc}
Every $2$-local $\frac{1}{2}$-derivation of $\mathcal{W}_2$ is a $\frac{1}{2}$-derivation.
\end{theorem}
\begin{proof}
Let $\Delta$ be an arbitrary 2-local $\frac{1}{2}$-derivation of $\mathcal{S}_2$. Then, by definition, for any element $a,\in \mathcal{W}_2$, there exists a $\frac{1}{2}$-derivation $D_{a,e_1}$ of $\mathcal{W}_2$ such that 
$$\Delta(a)=D_{a,\xi_1}(a), \, \mbox{and} \, \Delta(\xi_1)=D_{a,\xi_1}(\xi_1).$$
By theorem\ref{der2}, there exists a scalar $\lambda^{a,\xi_1}$ such that $D_{a,\xi_1}=\lambda^{a,\xi_i}id$.\\
Let $b$ be an arbitrary element in $\mathcal{W}_2$. Then, there exists a $\frac{1}{2}$-derivation of $\mathcal{W}_2$ such that 
$$\Delta(b)=D_{b,\xi_1}(b), \, \mbox{and} \, \Delta(\xi_1)=D_{b,\xi_1}(\xi_1).$$ 
By theorem\ref{der2}, there exists a scalar $\lambda^{b,\xi_1}$ such that $D_{b,\xi_1}=\lambda^{b,\xi_i}id$.\\ 
Since, $\Delta(\xi_1)=D_{a,\xi_1}(\xi_1)=D_{b,\xi_1}(\xi_1)$, we have
$$\lambda^{a,\xi_1}=\lambda^{b,\xi_1}$$
Thus 
$$D_{a,\xi_1}=D_{b,\xi_1}$$
Therefore, for any $a \in \mathcal{W}_2$
$$\Delta(a)=D_{b,\xi_1}(a)=\lambda^{b,\xi_1}a,$$
that is $D_{b,\xi_1}$ does not depend on $a$. Hence, $\Delta$ is a $\frac{1}{2}$-derivation. This ends the proof.
\end{proof}
\subsection{ Biderivations of $\mathcal{W}_2$}
Our main tool for study of biderivations of the algebras $\mathcal{W}(2)$  is the following theorem [\cite{Kayg1}, Theorem 8], where the matrix of a derivation is calculated in the new basis $\lbrace \xi_1,\xi_2,\dots,\xi_6 \rbrace$.  
\begin{theorem} \cite{Kayg1}
\label{lem2}
Let $d$ be a derivation of $\mathcal{W}(2)$. Then, the matrix $D_d$ of $d$ is of the form:
\begin{center}
$D_{d}= \begin{pmatrix}
		0 & a &0 &0&0&0\\
		0 & b &2a &0&0&0\\
		0 & 0 &2b &0&0&0\\
		-a & 0 &0 &-b&a&0\\
		0 & -a &0 &0&0&2a\\
		0 & 0 &-a &0&0&b
		\end{pmatrix}. $
\end{center}
where $a,b \in \mathbb{F}$.
\end{theorem}
Let $\delta$ be a  biderivation of  $\mathcal{W}_2$ and $x,y \in \mathcal{W}_2$, such that $x=\sum_{i=1}^{6}x_{i}\xi_{i}$ and $y=\sum_{i=1}^{6} y_{i}\xi_{i}$. Then, by the bilinearity of $\delta^-$, we obtain, 
$$\delta(x,y)=\sum_{i=1}^{6} \sum_{j=1}^{6} x_{i} y_{j} \delta(\xi_i,\xi_j)=
\sum_{i=1}^{6} \sum_{j=1}^{6} x_{i} y_{j} \delta_{\xi_i}(\xi_j).$$ 
By Theorem \ref{lem2}, the matrix $D_{\xi_i}$ of $\delta_{\xi_i}$, for $i=1,2,\dots,6$ in the basis $\lbrace \xi_1,\xi_2,\dots,\xi_6 \rbrace$ is of the form
\begin{equation}
\label{eq2}
D_{\xi_i}= \begin{pmatrix}
		0 & a_i &0 &0&0&0\\
		0 & b_i &2a_i &0&0&0\\
		0 & 0 &2b_i &0&0&0\\
		-a_i & 0 &0 &-b_i&a_i&0\\
		0 & -a_i &0 &0&0&2a_i\\
		0 & 0 &-a_i &0&0&b_i
		\end{pmatrix}. 
\end{equation}
We have the following theorem which characterize skew-biderivation of $\mathcal{W}_2$.
\begin{theorem}
\label{pro2}
The algebra $\mathcal{W}_2$ has no nontrivial skew-symmetric biderivation.
\end{theorem}
\begin{proof}
Let $\delta^-$ be a skew-symmetric biderivation of $\mathcal{W}_2$. Then, for $i=1,\dots , 6$ the matrix of $\delta^-_{\xi_i}$ is of the form (\ref{eq2}).\\ 
Since $ \delta^-$ is skew-symmetric, then, $\delta^-(\xi_i,\xi_i)=0$ for $i=1,2,\dots,6$. Therefore,
$$a_1=a_2=b_2=a_3=b_3=b_4=a_5=a_6=b_6=0.$$
For $i=1,4,5$, the equality $\delta^-(\xi_i,\xi_2)=-\delta^-(\xi_2\xi_i)$ implies that
$$b_1=a_4=b_5=0.$$
Thus, for $i=1, \dots,6$,
Thus, for $i=1, \dots,8$. 
\begin{center}
$D_{\xi_i}= \begin{pmatrix}
		0 & 0 &0 &0&0&0\\
		0 & 0 &0 &0&0&0\\
		0 & 0 &0 &0&0&0\\
		0 & 0 &0 &0&0&0\\
		0 & 0 &0 &0&0&0\\
		0 & 0 &0 &0&0&0
		\end{pmatrix}. $
\end{center}
Therefore, $\delta^-(x,y)=0$ for all $x,y \in \mathcal{W}_2$. 
\end{proof}
We give in the following theorem the characterization of symmetric biderivation of $\mathcal{W}_2$. 
\begin{theorem}
\label{pro2}
The algebra  $\mathcal{W}_2$ has no nontrivial symmetric biderivation.
\end{theorem}
\begin{proof}
Let $\delta^+$ be a skew-symmetric biderivation of $\mathcal{W}_2$. Then, for $i=1,\dots,6$ the matrix of $\delta^+_{\xi_i}$ is of the form (\ref{eq2}).\\
From the equality $\delta^+(\xi_2,\xi_i)=\delta^+(\xi_i,\xi_2)$ for $i=1,3,4,5,6$, we obtain
$$a_i=b_i=0, \quad i=1,\dots,6.$$
Thus, for any $1,\dots,6$
\begin{center}
$D_{\xi_i}= \begin{pmatrix}
		0 & 0 &0 &0&0&0\\
		0 &0 &0 &0&0&0\\
		0 & 0 &0 &0&0&0\\
		0 & 0 &0 &0&0&0\\
		0 & 0 &0 &0&0&0\\
		0 & 0 &0 &0&0&0
		\end{pmatrix}. $
\end{center}
That is, $\delta^+(x,y)=0$, for all $x,y \in \mathcal{W}_2$.
\end{proof}
Our main result of this section is now stated and proved in the following theorem.
\begin{theorem}
\label{thm2}
 $\delta$ is a biderivation of $\mathcal{W}_2$ if and only if $\delta$ is a zero mapping. i.e 
$BDer(\mathcal{W}_2)=\lbrace 0 \rbrace$.
\end{theorem}
\section{Terminal subalgebras $\mathcal{S}_2$ and $\mathcal{H}_1$ of the algebra $\mathcal{W}_2$.}
In \cite{Kayg1} et all show that Hence the algebra $\mathcal{H}_1$ has the
basis $\lbrace z_1,\dots , z_4\rbrace$ and the algebras $\mathcal{H}_1$ and $\mathcal{S}_2$ coincide with each other. The multiplication table of the terminal algebra \(\mathcal{S}_2\) is given by (see \cite{Kayg1}):

\[
\begin{array}{c|cccc}
\cdot & z_1 & z_2 & z_3 & z_4 \\
\hline
z_1 & -z_1 & z_2 & 3z_3 & -3z_4  \\
z_2 & -2z_2 & -3z_3 & 0 & z_1  \\
z_3 & 0 & 0 & 0 & 0  \\
z_4 & 3z_4 &-2z_1 & -z_2 & 0
\end{array}
\]

\subsection{$\frac{1}{2}$-derivations of $\mathcal{S}_2$}
In this section we characterize $\frac{1}{2}$-derivation and Local $\frac{1}{2}$-derivation (resp $2$-Local $\frac{1}{2}$derivation) of $\mathcal{S}_2$.
\begin{theorem}\label{der3}
Let $d:\mathcal{S}_2 \rightarrow \mathcal{S}_2$ be a linear map. Then $d$ is a $\frac{1}{2}$-derivation if and only if there exists a scalar $\lambda \in \mathbb{K}$ such that  
$$d(x)=\lambda x , \quad \forall x\in \mathcal{S}_2.$$
\end{theorem}
\begin{proof}
Let $d(z_j)=\sum a_{ij}z_i$ be a $\frac{1}{2}$-derivation of $\mathcal{S}_2$, where $a_{ij} \in \mathbb{K}$.\\
Since $-d(z_1)=d(z_1.z_1)=\frac{1}{2}(d(z_1).z_1+z_1.d(z_1))$, then $d(z_1)=a_{11}z_1$.\\
Since $d(z_2)=d(z_1.z_2)=\frac{1}{2}(d(z_1).z_2+z_1.d(z_2))$, then $d(z_2)=a_{11}z_2$.\\
Since $3d(z_3)=d(z_1.z_3)=\frac{1}{2}(d(z_1).z_3+z_1.d(z_3))$, then $d(z_3)=a_{11}z_3$.\\
Since $-3d(z_4)=d(z_1.z_4)=\frac{1}{2}(d(z_1).z_4+z_1.d(z_4))$, then $d(z_4)=a_{11}z_4$.
\end{proof}
\begin{corollary}
Let $\mathcal{S}_2$ be the terminal subalgebra of $\mathcal{W}_2$, then
$$\Gamma(\mathcal{W}_2)=\lbrace \lambda id \, :\, \lambda \in \mathbb{K} \rbrace.$$ 
\end{corollary}

Now, we are going to prove that any local (resp. $2$-local) $\frac{1}{2}$-derivation of $\mathcal{S}_2$ is a $\frac{1}{2}$-derivation.

\begin{theorem}
Every local $\frac{1}{2}$-derivation of  $\mathcal{S}_2$ is a $\frac{1}{2}$-derivation.
\end{theorem}
\begin{proof}

Let $D$ be a local $\frac{1}{2}$-derivation, for all $x \in \mathcal{S}_2$ there exists a scalar $\lambda_x$ such that $D(x)=\lambda_x x$. \\
Taking, $x=z_1,\dots,z_4$ we have $D(z_i)=\lambda_{z_i} z_i$, then it sufficient to show that $\lambda_{z_i}=\lambda_{z_j}$ for $i,j=1,\dots,4$.\\
Since $D$ is linear, then $D(z_i+z_j)=D(z_i)+D(z_j)$. Which implies that $\lambda_{z_i}=\lambda_{z_i+z_j}=\lambda_{z_j}$. So $D(z_i)=\lambda z_i$ for all $i=1,\dots,4$. Therefore, by theorem\ref{der3}, $D$ is a $\frac{1}{2}$-derivation of $\mathcal{S}_2$. This ends the proof.
\end{proof}

\begin{theorem}
Every 2-Local $\frac{1}{2}$-derivation of  $\mathcal{S}_2$ is a $\frac{1}{2}$-derivation.
\end{theorem}
\begin{proof}

Let $\Delta$ be an arbitrary 2-local $\frac{1}{2}$-derivation of $\mathcal{S}_2$. Then, by definition, for any element $a,\in \mathcal{S}_2$, there exists a $\frac{1}{2}$-derivation $D_{a,z_1}$ of $\mathcal{S}_2$ such that 
$$\Delta(a)=D_{a,z_1}(a), \, \mbox{and} \, \Delta(z_1)=D_{a,z_1}(z_1).$$
By theorem\ref{der3}, there exists a scalar $\lambda^{a,z_1}$ such that $D_{a,z_1}=\lambda^{a,z_i}id$.\\
Let $b$ be an arbitrary element in $\mathcal{S}_2$. Then, there exists a $\frac{1}{2}$-derivation of $\mathcal{S}_2$ such that 
$$\Delta(b)=D_{b,z_1}(b), \, \mbox{and} \, \Delta(z_1)=D_{b,z_1}(z_1).$$ 
By theorem\ref{der3}, there exists a scalar $\lambda^{b,z_1}$ such that $D_{b,z_1}=\lambda^{b,z_i}id$.\\ 
Since, $\Delta(z_1)=D_{a,z_1}(z_1)=D_{b,z_1}(z_1)$, we have
$$\lambda^{a,z_1}=\lambda^{b,z_1}$$
Thus 
$$D_{a,z_1}=D_{b,z_1}$$
Therefore, for any $a \in \mathcal{S}_2$,
$$\Delta(a)=D_{b,z_1}(a)=\lambda^{b,z_1}a,$$
that is $D_{b,z_1}$ does not depend on $a$. Hence, $\Delta$ is a $\frac{1}{2}$-derivation. This ends the proof.
\end{proof}
\subsection{ Biderivations of $\mathcal{S}_2$}
Our main tool for study of biderivations of the algebras $\mathcal{S}_2$  is the following theorem [\cite{Kayg1}, Theorem 10], where the matrix of a derivation is calculated in the new basis $\lbrace z_1,z_2,z_3,z_4 \rbrace$.  
\begin{theorem} \cite{Kayg1}
\label{lem3}
Let $d$ be a derivation of $\mathcal{S}_2$. Then, the matrix $D_d$ of $d$ is of the form:
\begin{center}
$D_{d}= \begin{pmatrix}
		0 & 0 &0 &b\\
		-2b & a &0 &0\\
		0 & -3b &2a &0\\
		0 & 0 &0 &-a
		\end{pmatrix}. $
\end{center}
where $a,b \in \mathbb{F}$.
\end{theorem}
Let $\delta$ be a  biderivation of  $\mathcal{S}_2$ and $x,y \in \mathcal{S}_2$, such that $x=\sum_{i=1}^{4}x_{i}z_{i}$ and $y=\sum_{i=1}^{4} y_{i}z_{i}$. Then, by the bilinearity of $\delta$, we obtain, 
$$\delta(x,y)=\sum_{i=1}^{4} \sum_{j=1}^{4} x_{i} y_{j} \delta(z_i,z_j)=
\sum_{i=1}^{4} \sum_{j=1}^{4} x_{i} y_{j} \delta_{z_i}(z_j).$$ 
By Theorem \ref{lem3}, the matrix $D_{z_i}$ of $\delta_{\xi_i}$, for $i=1,\dots,4$ in the basis $\lbrace z_1,z_2,z_3,z_4 \rbrace$ is of the form
\begin{equation}
\label{eq3}
D_{d}= \begin{pmatrix}
		0 & 0 &0 &b_i\\
		-2b_i & a_i &0 &0\\
		0 & -3b_i &2a_i &0\\
		0 & 0 &0 &-a_i
		\end{pmatrix}. 
\end{equation}
Now, we characterize skew-symmetric biderivation of $\mathcal{S}_2$.
\begin{theorem}
\label{pro1}
The algebra $\mathcal{S}_2$ has no nontrivial skew-symmetric biderivation.
\end{theorem}
\begin{proof}
Let $\delta^-$ be a skew-symmetric biderivation of $\mathcal{S}_2$. Then, for $i=1,\dots,4$ the matrix of $\delta^-_{z_i}$ is of the form (\ref{eq3}).\\ 
Since $ \delta^-$ is skew-symmetric, then, $\delta^-(z_i,z_i)=0$ for $i=1,2,3,4$. Therefore,
$$b_1=a_2=b_2=a_3=a_4=b_4=0.$$
For $i=1,3$, the equality $\delta^-(z_i,z_2)=-\delta^-(z_2,z_i)$ implies that
$$a_1=b=3=0.$$
That is, for $i=1,2,3,4$. 
\begin{center}
$D_{z_i}= \begin{pmatrix}
		0 & 0 &0 &0\\
		0 & 0 &0 &0\\
		0 & 0 &0 &0\\
		0 & 0 &0 &0
		\end{pmatrix}. $
\end{center}
Therefore, $\delta^-(x,y)=0$ for all $x,y \in \mathcal{S}_2$.
\end{proof}
Symmetric biderivation of $\mathcal{S}_2$ are characterize in the following theorem.
\begin{theorem}
\label{pro2}
The algebra $\mathcal{S}_2$ has no nontrivial symmetric biderivation.
\end{theorem}
\begin{proof}
Let $\delta^+$ be a skew-symmetric biderivation of $\mathcal{S}_2$. Then, for $i=1,2,3,4$ the matrix of $\delta^+_{z_i}$ is of the form (\ref{eq3}).\\
For $i=1,2,3$ the equality $\delta^+(z_i,z_4)=\delta^+(z_4,z_i)$ we deduce,
$$a_i=b_i=0,\quad i=1,2,3,4.$$
That is, for $i=1,2,3,4$. 
\begin{center}
$D_{z_i}= \begin{pmatrix}
		0 & 0 &0 &0\\
		0 & 0 &0 &0\\
		0 & 0 &0 &0\\
		0 & 0 &0 &0
		\end{pmatrix}. $
\end{center}
Therefore, $\delta^+(x,y)=0$ for all $x,y \in \mathcal{S}_2$.
\end{proof}
Our main result of this section is now stated and proved in the following theorem.
\begin{theorem}
\label{thm2}
 $\delta$ is a biderivation of $\mathcal{S}_2$ if and only if $\delta$ is a zero mapping. i.e 
$BDer(\mathcal{S}_2)=\lbrace 0 \rbrace$.
\end{theorem}

\vspace{1cm}
\noindent \textbf{Acknowledgements:} The authors thank the referees for their valuable comments that contributed to a sensible improvement of the paper.

\section*{Declarations}
\noindent \textbf{Data Availability Statement} No datasets were generated or analysed during the current study\\ 
\noindent \textbf{Conflicts of interest:} The authors declare that they have no conflict of interest.


\begin{thebibliography}{99}
\bibitem{Arz}
 Arzikulov, F. and  Umrzaqov, N. “Conservative Algebras of 2-Dimensional Algebras, III”,
Communications in Mathematics 29 (2021), no. 2, DOI 10.2478/cm-2021-0023.
\bibitem{Baj}
Bajo, I., Benayadi, S. and Oubba, H., 2025. Non-existence of Symmetric Biderivations on Finite-Dimensional Perfect Lie Algebras: I. Bajo et al. Mediterranean Journal of Mathematics, 22(6), p.162.
\bibitem{bn}
Benayadi, S.  Oubba, H. Nonassociative algebras of biderivation-type. Linear Algebra and its Applications, 2024, vol. 701, p. 22-60.

\bibitem{br0}
 Bresar, M.  Martindale, W. S.  Miers, C. R. Centralizing maps in prime rings with involution, J. Algebra 161, 342-357 (1993).
 \bibitem{br1}
   Bresar, M.  Zhao, K. Biderivations and commuting linear maps on Lie algebras, J. Lie Theory 28,
885-900 (2018).
\bibitem{du}
 Du, Y. Wang, Y.  Biderivations of generalized matrix algebras, Linear Algebra Appl. 438(11), 4483–
4499 (2013).


\bibitem{Kad}
Kadison, R.V. Local derivations, Journal of Algebra, 130, (1990), 494-509.




\bibitem{Kan1}
Kantor, I.L. Certain generalizations of Jordan algebras (Russian). Trudy Sem. Vector. Tenzor. Anal. 16(1972) 407-499.

\bibitem{Kan2}
Kantor, I.L. The universal conservative algebra. Siberian Mathematical Journal 31 (3)(1990) 388-395.

\bibitem{Kan3}
 Kantor, I. On an extension of a class of Jordan algebras, Algebra Logic 28 (2) (1989) 117-121.

\bibitem{Kayg1}
Kaygorodov, I. Lopatin, A. Popov, Yo. Conservative algebras of 2-dimensional algebras. Linear Algebra and its Applications 486 (2015) 255-274.
\bibitem{Kay}
 Kaygorodov, I, On the Kantor product, arXiv:1506.00736.
 \bibitem{Iv1}
 Kaygorodov, I.  Volkov, Yo.: Conservative algebras of 2-dimensional algebras, II.
Communications in Algebra 45 (8) (2017) 3413–3421.
\bibitem{Iv2}
Kaygorodov I. Khudoyberdiyev A. Sattarov A.: One-generated nilpotent terminal
algebras. Communications in Algebra 48 (10) (2020) 4355–4390.
\bibitem{Iv3}
I. Kaygorodov, M. Khrypchenko, Yu. Popov: The algebraic and geometric classification
of nilpotent terminal algebras. Journal of Pure and Applied Algebra 225 (6) (2021)
106625.
 \bibitem{Khu}
  Khudoyberdiyev, A. KH. Yusupov, B. "Local and 2-local 1 2 -derivation on finite-dimensional Lie algebras", Results in Mathematics, 79 (2024), 210.
 \bibitem{Lar}
 Larson, D.R. Sourour, A.R Local derivations and local automorphisms of B(X), Proceedings of Symposia in Pure Mathematics, 51 (1990), 187-194.
 \bibitem{Mama}
 Mamadaliyev, U.X., Sattarov, A., and Yusupov, B. (2025). LOCAL AND 2-LOCAL $\frac{1}{2}$-DERIVATIONS OF SOLVABLE LEIBNIZ ALGEBRAS. EURASIAN MATHEMATICAL JOURNAL.

\bibitem{ob0}
Oubba, H. Generalized quaternion algebras. Rendiconti del Circolo Matematico di Palermo Series 2, 2023, vol. 72, no 8, p. 4239-4250.
\bibitem{ob1}
Oubba, H. Biderivations, commuting linear maps, post-Lie algebra structure on solvable Lie algebras of maximal rank. Communications in Algebra, 2024, p. 1-10.
\bibitem{ob2}
Oubba, H. Local (2-Local) derivations and automorphisms and biderivations of complex $\omega$-Lie algebras. Le Matematiche, 2024, vol. 79, no 1, p. 135-150.
\bibitem{Oub}
Oubba, H., 2026. Biderivations of some evolution algebras. Asian-European Journal of Mathematics, p.2650050.







\bibitem{Sem}
 Semrl P., Local automorphisms and derivations on  B(H), Proceedings of the American Mathematical Society, 125 (1997), 2677-2680.









\end{thebibliography}
\end{document}